\documentclass{amsart}

\usepackage{amsmath,amssymb,amsthm}

\newtheorem{prop}{Proposition}[section]
\newtheorem{thm}[prop]{Theorem}
\newtheorem{lem}[prop]{Lemma}
\newtheorem{cor}[prop]{Corollary}

\newtheorem{defn}[prop]{Definition}

\theoremstyle{definition}

\newtheorem{rem}[prop]{Remark}

\newtheorem*{ack}{Acknowledgement}

\def\co{\colon\thinspace}

\newcommand{\C}{\mathbb C}

\newcommand{\rmd}{\mathrm d}

\newcommand{\rme}{\mathrm e}

\newcommand{\rmi}{\mathrm i}

\newcommand{\MM}{\mathcal M}

\newcommand{\N}{\mathbb N}

\newcommand{\Q}{\mathbb Q}

\newcommand{\R}{\mathbb R}

\renewcommand{\SS}{\mathcal S}

\newcommand{\TT}{\mathcal T}

\newcommand{\UU}{\mathcal U}

\newcommand{\WW}{\mathcal W}

\newcommand{\Z}{\mathbb Z}
\newcommand{\ZZ}{\mathcal Z}

\newcommand{\CR}[1]{\bar{\partial}_{#1}}

\newcommand{\lra}{\longrightarrow}
\newcommand{\ra}{\rightarrow}

\DeclareMathOperator{\cut}{\mathrm{cut}}

\DeclareMathOperator{\ev}{\mathrm{ev}}

\DeclareMathOperator{\fs}{\mathrm{FS}}

\DeclareMathOperator{\genus}{\mathrm{genus}}

\DeclareMathOperator{\Hom}{\mathrm{Hom}}

\DeclareMathOperator{\loc}{\mathrm{loc}}

\newcommand{\boldxi}{\mbox{\boldmath $\xi$}}

\begin{document}

\author{Stefan Suhr}
\address{DMA, \'Ecole Normale sup\'erieure and Universit\'e Paris-Dauphine,
45 rue d'Ulm, 75230 Paris Cedex 05, France}
\email{stefan.suhr@ens.fr}
\author{Kai Zehmisch}
\address{Mathematisches Institut, Westf\"alische Wilhelms-Universit\"at M\"unster, Einsteinstr. 62,
D-48149 M\"unster, Germany}
\email{kai.zehmisch@wwu.de}

\title[Polyfolds and the Weinstein conjecture]
{Polyfolds, Cobordisms, and the strong Weinstein conjecture}

\date{13.10.2016}

\begin{abstract}
  We prove the strong Weinstein conjecture
  for closed contact manifolds
  that appear as the concave boundary
  of a symplectic cobordism
  admitting an essential local foliation by holomorphic spheres.
\end{abstract}

\thanks{The research in this article is supported by the Institute for Advanced Study (Princeton, NJ).
S. Suhr: The research leading to these results has received funding from the
European Research Council under the European Union's Seventh Framework
Programme (FP/2007-2013) / ERC Grant Agreement 307062.
K. Zehmisch is partially supported by DFG grant ZE 992/1-1.}

\maketitle

%%%%%%%%%%%%%%%%%%%%%%%%%%%%%%%%%%%
%%%%%%%%%%%%%%%%%%%%%%%%%%%%%%%%%%%

\section{Introduction\label{sec:intro}}

Given a closed (co-orientable) contact manifold $(M,\xi)$
and a defining contact form $\alpha$
(i.e.\ $\xi=\ker\alpha$)
Weinstein conjectured in \cite{wein79}
that the Reeb vector field $R$,
which is uniquely defined by
$i_R\rmd\alpha=0$ and $\alpha(R)=1$,
admits a periodic solution.
The Weinstein conjecture was proven by Taubes \cite{taub07}
for all closed contact manifolds of 
dimension $3$
and has been verified in higher dimensions
in many situations most recently in the presence of
contact connected sums in \cite{geizeh14a,geizeh14b,gnw14}.
We refer the reader to \cite{pasq12}
for the state of the art of the conjecture.

A stronger conjecture was given in \cite{ach05}
that asks for a finite collection of periodic solutions of $R$,
a so-called {\bf null-homologous Reeb link},
that oriented by $R$ and eventually counted with a positive multiple of a period
represents the trivial class in the homology of $M$.
We refer to the existence question of a null-homologous Reeb link
as the {\bf strong Weinstein conjecture}
and remark
that the stronger version of the conjecture is not covered
by Taubes result.
The aim of the present work
is to confirm the strong Weinstein conjecture
for closed contact manifolds $(M,\xi=\ker\alpha)$
that appear as the concave end of symplectic cobordisms
with particular properties.
This will generalize the results obtained in \cite{geizeh12}.

The notion of a symplectic cobordism was introduced in
\cite{egh00,hwz03} in the context of symplectic field theory.
A {\bf symplectic cobordism}
is a compact connected symplectic manifold $(W,\omega)$
with boundary that admits a {\bf Liouville vector field} $Y$
near $\partial W$,
which by definition satisfies $L_Y\omega=\omega$.
According to the boundary orientation induced
by the orientation of the symplectic form
the Liouville vector field $Y$ points either in or out of $W$.
This decomposes the boundary of $W$ into the {\bf concave boundary}
$M_-$ (along which $Y$ points in) and into the {\bf convex boundary}
$M_+$ (along which $Y$ points out).
The Liouville vector field defines contact forms
$\alpha_-$ and $\alpha_+$ by restricting
$i_Y\omega$ to the tangent bundles of $M_-$ and $M_+$, resp.,
so that $(M_-,\alpha_-)$ and $(M_+,\alpha_+)$
are particular examples of contact type hypersurfaces
in $(W,\omega)$, cf. \cite{mcsa98}.

A theorem of Hofer \cite{hofe93} relates
the existence of closed Reeb orbits
to the existence of (punctured finite energy) holomorphic curves
in symplectic cobordisms, cf.\ \cite{albhof09,geizeh12,geizeh13}.
In order to utilize this relation we will consider symplectic
cobordisms that admit a so-called {\it essential local foliation
by holomorphic spheres}.
The definition is given in Section \ref{subsec:essholfol} below.

Generalizing the notion of symplectic cobordisms
we will consider compact connected symplectic
manifolds $(W,\omega)$
that have boundary components
that are either concave, convex, or
foliated by symplectic spheres.
In the case the foliation is given
by an essential local foliation
by holomorphic spheres
in the sense of Section \ref{subsec:essholfol}
$(W,\omega)$ is called a {\bf symplectic cobordism}
as well.

The symplectic area of the holomorphic spheres of the foliation
as it will turn out induces an upper bound
on the  {\bf total action} of a null-homologous Reeb link,
which is by definition
the sum of the actions of the link components
counted with the selected period multiplicities.

\begin{thm}
  \label{thm:mainthm}
 If $(M,\alpha)$ is the concave boundary
 of a symplectic cobordism $(W,\omega)$,
 and $(W,\omega)$
 admits an essential local foliation
 by holomorphic spheres of area $\pi\varrho^2$,
 then there exists a null-homologous Reeb link
 in $(M,\alpha)$ of total action smaller than $\pi\varrho^2$.
 Additionally,
 if $(W,\omega)$ has no concave boundary
 it has no convex boundary either.
\end{thm}

\begin{rem}
 The only surface $(W,\omega)$
 to which the theorem applies is $\C P^1$.
\end{rem}

The qualitative content of the theorem is valid not only
for one particular contact form on $M$.
In fact, the construction from \cite[Section 3.3]{geizeh12}
allows one to conclude
for any contact form whose kernel is equal to
$\xi=\ker\alpha$.
Each contact form that defines $(M,\xi)$
appears as a graph over the zero section
in the symplectisation of $(M,\alpha)$.
After a shift in the negative $\R$-direction the graph
can be assumed to lie below the zero section.
If $(M,\alpha)$ is the concave boundary
of a symplectic cobordism $(W,\omega)$,
then a positive constant multiple of
any $\xi$-defining contact form
can be realized 
as the concave boundary of a slightly modified
symplectic cobordism.
The cobordism is obtained by
gluing the symplectic cobordisms that is
cut out by the shifted graph and the zero section
to the cobordism $(W,\omega)$ along $(M,\alpha)$.
To express this circumstance we will say that 
$(M,\xi)$ is the concave boundary of $(W,\omega)$.
Convex boundaries are handled similarly.

\begin{cor}
 The strong Weinstein conjecture holds for all contact manifolds
 that are the concave boundary of a symplectic cobordism
 that is provided with an essential local foliation by holomorphic spheres.
\end{cor}

A contact manifold is called {\bf co-fillable}
if it is a boundary component of a symplectic cobordism
that has more then one convex but no concave boundary components,
cf.\ \cite{wen13}.
Examples of McDuff \cite{mcd91}, Geiges \cite{gei94,gei95},
and Massot-Niederkr\"uger-Wendl \cite{mnw13} show the
existence of co-fillabe contact manifolds.
Generalizing a result of McDuff \cite{mcd91} we obtain
(cf.\ Remark \ref{rem:weak}):

\begin{cor}
 If $(M,\xi)$ is the concave boundary of a
 symplectic cobordism $(W,\omega)$
 as in the theorem, then $(M,\xi)$ is not co-fillable.
\end{cor}

In \cite{geizeh12} the following capacity
for symplectic manifolds $(V,\omega)$
that are not of dimension $2$
was introduced:
\[
c(V,\omega)=
\sup_{(M,\alpha)}\inf\alpha\;.
\]
The supremum is taken over all closed contact type hypersurfaces 
$(M,\alpha)$ in $(V,\omega)$ and $\inf\alpha$ is the minimal
total action of a null-homologous Reeb link in $(M,\alpha)$.

\begin{cor}
 If $(W,\omega)$ as in the theorem
 has no concave (and hence no convex) boundary,
 then $c(W,\omega)\leq\pi\varrho^2$.
\end{cor}

In particular, the Gromov radius
of $(W,\omega)$ as introduced in \cite{grom85}
is bounded by $\pi\varrho^2$
from above.
This can also be derived from the following {\it uniruledness} result.
This is
because no embedding of an open ball into $W$
can intersect $\partial B\times\C P^1$,
which is the boundary (in case if non-empty)
of the local foliation domain $B\times\C P^1$
that we require to exist,
cf.\ Section \ref{subsec:essholfol}.

\begin{cor}
  \label{cor:gwstatement}
 If $(W,\omega)$ as in the theorem
 has no contact type boundary components,
 then through each point of $W$ there passes a
 (nodal) holomorphic sphere for any compatible
 almost complex structure that coincides with
 $J_B\oplus\rmi$ on $U\times\C P^1$,
 where $U$ is a collar neighborhood of $\partial B$
 in $B$.
\end{cor}

In order to prove the theorem
we use holomorphic spheres corresponding to the
given local foliation.
The associated moduli space is non-empty
and regular in a neighborhood of the holomorphic spheres
that come from the local foliation.
The assumption of being essential results in
uniqueness properties of the moduli space.
In order to ensure global regularity properties of the moduli space,
which are obstructed by bubbling off of multiply covered
holomorphic spheres of negative first Chern class,
semi-positivity of $(W,\omega)$ could be used.
In order to get an unrestricted statement we will employ
the regularity theory developed by Hofer-Wysocki-Zehnder instead.
The space of (not necessarily holomorphic)
stable curves has a polyfold structure,
see \cite{hwz-gw11}.
Using abstract perturbations (see \cite{hwz-III09})
the moduli space can be approximated
by solution spaces of perturbed Fredholm problems
that carry the structure of a smooth branched orbifold
with weights.
This enables us to conclude as in \cite{geizeh12}. 

In Section \ref{sec:definition}
we formulate the definition of an
essential local foliation by holomorphic spheres.
Examples and applications are presented
in Section \ref{sec:exandappl}.
The content of Section \ref{sec:stable}-\ref{sec:polystr}
is a description of the moduli space of holomorphic curves
relevant in our situation in view of the application of the
theory of polyfolds and abstract perturbations.
The proof of Theorem \ref{thm:mainthm}
is given in Section \ref{sec:proof}.

%%%%%%%%%%%%%%%%%%%%%%%%%%%%%%%%%%%
%%%%%%%%%%%%%%%%%%%%%%%%%%%%%%%%%%%

\section{Definition\label{sec:definition}}

%%%%%%%%%%%%%%%%%%%%%%%%%%%%%%%%%%%%%%%%%%%%

\subsection{Completion\label{subsec:compl}}

A collar neighborhood of a concave boundary $(M,\alpha)$
of a symplectic cobordism $(W,\omega)$
is symplectomorphic to
$
\big([0,\varepsilon)\times M,\rmd(\rme^{a}\alpha)\big)
$
for some $\varepsilon>0$.
Let $\TT$ denote the set of all smooth
strictly increasing functions $(-\infty,0]\ra(0,1]$
that allow a smooth extension to $\R$
that restricted to $[0,\infty)$ coincides with $a\mapsto\rme^{a}$.
A {\bf concave end} is a symplectic manifold of the form
$
\big((-\infty,0]\times M,\rmd(\tau\alpha)\big)
$
with $\tau\in\TT$.
Gluing the concave end to $(W,\omega)$ along $(M,\alpha)$
results in a symplectic manifold $(W',\omega_{\tau})$
for $\tau\in\TT$.
In order to reflect the conformal nature we will call
$(W',\omega_{\tau})$
{\it the} {\bf completion} of $(W,\omega)$.

%%%%%%%%%%%%%%%%%%%%%%%%%%%%%%%%%%%%%%%%%%%%

\subsection{Holomorphic spheres with nodes\label{subsec:holsphnod}}

Consider a Riemann surface $(S,j)$
with finitely many connected components.
A {\bf nodal pair} is a subset of $S$
that consists of two distinct points.
Nodal pairs are required to be pairwise disjoint.
Denote by $D$ a finite set of nodal pairs
such that the quotient space
$S/D$ is connected.
A {\bf nodal Riemann surface} is a triple $(S,j,D)$.
The nodal Riemann surface $(S,j,D)$
is said to be of {\bf genus zero}
if each connected component of $S$
is diffeomorphic to the $2$-sphere
and if the number of connected components of $S$
is equal to the number of nodal pairs $\#D$ plus one;
to phrase it differently, if $S/D$ is simply connected.

Let $J$ denote an almost complex structure on the completion of $W$.
We consider a nodal Riemann surface $(S,j,D)$ of genus zero.
A {\bf nodal} $J${\bf-holomorphic sphere}
is a smooth map $u\co S\ra W'$
that solves the non-linear Cauchy-Riemann equation
$Tu\circ j=J(u)\circ Tu$
and that descends
to a continuous map on the quotient $S/D$.
In case $D$ is empty $u$ is called {\bf un-noded}.

%%%%%%%%%%%%%%%%%%%%%%%%%%%%%%%%%%%%%%%%%%%%

\subsection{Essential holomorphic foliations\label{subsec:essholfol}}

Let $(B,\omega_B)$ be an open symplectic manifold.
If $B$ has non-empty boundary we require
that $\partial B$ is closed.
Denote by $\omega_{\fs}$ the Fubini-Study form on $\C P^1$,
which integrates to total area $\pi$.
A {\bf local foliation by holomorphic spheres} in $(W,\omega)$
is a symplectic embedding
\[
\big(B\times\C P^1,\omega_B+\varrho^2\omega_{\fs}\big)\lra(W,\omega)\;,
\]
where  $\varrho$ is a positive real number.
If the boundary of $B$ is non-empty we assume
that $\partial B\times\C P^1$ is mapped diffeomorphically onto
$\partial W\setminus (M_-\cup M_+)$
assuming that besides the concave and the convex boundary
a further boundary component ({\it a posteriory}
equal to $\partial B\times\C P^1$) exist.

The local foliation $B\times\C P^1$
is assumed to be equipped with a compatible
almost complex structure of the form $J_B\oplus\rmi$.
All almost complex structures
on the completion $W'$ 
are assumed to restrict to the
split structure $J_B\oplus\rmi$
on $B\times\C P^1$ and are called {\bf admissible}.

We remark that by a theorem of Moser \cite{mos65}
for any positive area form $\sigma$ on $S^2$
with $\sigma(S^2)=\pi\varrho^2$
there exists a diffeomorphism of $S^2$
along which $\sigma$ pulls back to $\varrho^2\omega_{\fs}$.
The {\bf area} of a holomorphic sphere in the
local foliation is $\pi\varrho^2$.

\begin{defn}
  \label{essdef}
  A local foliation $B\times\C P^1\subset W$ as readily defined
  is called {\bf essential} if there exists a point
  $*\in B$ such that
  for all admissible compatible
  almost complex structures $J$ on $W'$
  any nodal $J$-holomorphic sphere in $W'$
  that is
  \begin{itemize}
  \item 
   homologous to $*\times\C P^1$,
  \item 
   non-constant restricted to any component of $S$,
  \item 
   and intersects $B\times \C P^1$ non-trivially
\end{itemize}
  is un-noded and up to a pre-composition
  with a M\"obius transformation equal to
  $z\mapsto (b,z)$ for some $b\in B$.
\end{defn}

%%%%%%%%%%%%%%%%%%%%%%%%%%%%%%%%%%%%%%%%
%%%%%%%%%%%%%%%%%%%%%%%%%%%%%%%%%%%%%%%%

\section{Examples and Applications\label{sec:exandappl}}

In \cite{geizeh12} the strong Weinstein conjecture
was verified
for contact manifolds that appear as the concave boundary
of a semipositive symplectic cobordism
that can be capped off along a convex boundary component
in a particular way.
The construction of the cap from \cite[Section 5.1]{geizeh12} generalizes
to the present context as follows.

%%%%%%%%%%%%%%%%%%%%%%%%%%%%%%%%%%%%%%%%%%%%

\subsection{Holomorphic foliations via caps\label{subsec:holfolcap}}

We consider a non-empty closed contact manifold
$(N,\alpha_N)$.
For $\varepsilon>0$ sufficiently small we denote by
\[
(V,\omega_V)=\big((-\varepsilon,0]\times N,\rmd(\rme^{a}\alpha_N)\big)
\]
a cylindrical subset of the symplectization of $(N,\alpha_N)$.
We equip $(V,\omega_V)$ with an almost complex structure
$J_V$ that is {\bf compatible with the contact form} $\alpha_N$,
i.e.\ $J_V$ is invariant under translation,
restricts to a compatible complex structure
on the symplectic bundle $(\ker\alpha_N,\rmd\alpha_N)$,
and sends the Liouville vector field $\partial_a$
to the Reeb vector field of $\alpha_N$.

Moreover, let $(Q,\omega_Q)$ be a closed
symplectic manifold and
denote by $J_Q$ a compatible
almost complex structure on $(Q,\omega_Q)$.
The {\bf rational area spectrum}
of the symplectic form $\omega_Q$
is given by the image of the map
$H_2(Q;\Q)\ra\R$
obtained by integration against $\omega_Q$.
The spectrum is countable and hence
constitutes a residual subset of $\R$.

\begin{prop}
\label{prop:essential}
 Let $(W,\omega)$ be a symplectic manifold
 with boundary and consider $(V,\omega_V)$
 and $(Q,\omega_Q)$
 as stated above.
 We assume that $(W,\omega)$
 admits a local foliation
 \[
 \Big(V\times Q\times\C P^1,\omega_V+\omega_Q+\varrho^2\omega_{\fs}\Big)
 \]
 with
 $\partial W=N\times Q\times\C P^1$
 that is equipped with the almost complex structure
 \[J_V\oplus J_Q\oplus\rmi\;.\]
 The local foliation is essential
 provided $\pi\varrho^2$ is not a
 rational spectral value of $\omega_Q$.
\end{prop}

\begin{proof}
 Let $u\co S\ra W$ be a nodal holomorphic sphere
 that intersects the local foliation $V\times Q\times\C P^1$
 non-trivially.
 The restriction of $u$ to $u^{-1}\big(V\times Q\times\C P^1\big)$
 can be projected to the $(-\varepsilon,0]$-factor of $V$.
 The composition of the resulting map
 with the exponential map
 is subharmonic and has an interior maximum.
 By the maximum principle and an open-closed argument
 applied to each component of $S$ the image $u(S)$
 is contained in $\{a\}\times N\times Q\times\C P^1$
 for a suitable $a\in(-\varepsilon,0]$.
 Because the projection of $u$ to the
 exact symplectic manifold $(-\varepsilon,0]\times N$
 must be constant (by compatibility)
 the image of $u$ is in fact
 contained in $\{v\}\times Q\times\C P^1\equiv Q\times\C P^1$
 for a suitable $v\in V$.

 In view of the Definition \ref{essdef}
 we assume in addition that $u$ is homologous
 to $\C P^1\equiv*\times\C P^1$ in $W$
 for any choice of base point $*$ of $V\times Q$
 and that $u$ is non-constant
 on each component of $S$.
 Choose an ordering on the components of
 $S=S_1\sqcup\ldots\sqcup S_k$
 and write $u_j$ for the restriction of $u$ to the component $S_j$.
 Further, denote by $u^Q_j$ and $\varphi_j$
 the projections of $u_j$ to $Q$ and $\C P^1$,
 resp.
 According to K\"unneth's formula with respect to $Q\times\C P^1$
 we get
 \[
 \big[\C P^1\big]=
 \sum_{j=1}^k\big[u^Q_j\big]+
 \sum_{j=1}^kd_j\big[\C P^1\big]
 \]
 in $H_2W$,
 where $d_j$ is the degree of the holomorphic map $\varphi_j$.
 Because $\pi\varrho^2$ is not in the rational area spectrum of $\omega_Q$
 an application of $\omega_W$ shows that
 \[
 \sum_{j=1}^k\omega_Q\Big(\big[u^Q_j\big]\Big)=0
 \qquad\text{and}\qquad
 \sum_{j=1}^kd_j=1\;.
 \]
 Because the symplectic energy is non-negative by compatibility
 we get $k=1$, $d_1=1$, and $u^Q_1$ is constant.
 Therefore,
 there exists $q\in Q$ and an automorphism $\varphi$
 of $\C P^1$ such that $u(z)=\big(v,q,\varphi(z)\big)$
 for all $z\in\C P^1$.
\end{proof}

\begin{rem}
 The construction generalizes to a
 symplectic manifold $(V,\omega_V)$
 that has a weakly convex contact type boundary,
 see Remark \ref{rem:weak},
 or is the negative half-symplectisation
 of the stable Hamiltonian structure
 $(\omega_N,\rmd\theta)$
 that is induced by a symplectic fibration $\theta\co N\ra S^1$
 on the boundary $N$ of $V$
 with respect to $\omega_N:=\omega_V|_{TN}$,
 see \cite[p.~877]{dgz14}.
 In the second case one requires that $\pi\varrho^2$
 is not a rational area spectral value of
 $\omega_F+\omega_Q$,
 where $F$ is the typical fibre of $\theta$,
 which is symplectic with respect to
 $\omega_F:=\omega_V|_{TF}$.
 The compatible almost complex structures
 $J_V$ we are now considering
 are translation invariant,
 turn the fibres of $\theta$ into a
 holomorphic submanifold
 in each level of $(-\varepsilon,0]\times N$,
 and send the Reeb vector field of
 the stable Hamiltonian structure
 $(\omega_N,\rmd\theta)$
 to $\partial_a$,
 see \cite[Section 5]{dgz14}.
\end{rem}

%%%%%%%%%%%%%%%%%%%%%%%%%%%%%%%%%%%%%%%%%%%%

\subsection{Stabilized Weinstein conjecture\label{subsec:stabwein}}

In the following we will construct a class
of symplectic cobordisms
with an essential local foliation as described.
Let $(P,\omega_P)$ be a {\bf symplectic filling},
i.e.\ a symplectic cobordism
with all boundary components
of contact type being convex.
Consider a closed hypersurface
\[
\Sigma\subset P\times Q\times\C
\]
that is of contact type with respect to
$\omega_P+\omega_Q+\rmd x\wedge\rmd y$.
The induced contact form on $\Sigma$
is denoted by $\alpha_{\Sigma}$.
Each component of $\Sigma$
bounds a relatively compact open domain
the so-called {\bf bounded domain}.
We require that the bounded domains are pairwise disjoint.
We denote the union of the domains by $D_{\Sigma}$.

Let $\varrho$ be a positive real number
such that $\pi\varrho^2$ is not in the rational
area spectrum of $(Q,\omega_Q)$
and greater than
the minimal area of a closed disc in $\C$
about the origin that contains the image
of the projection map
$\Sigma\subset P\times Q\times\C\ra\C$.
We define the {\bf symplectic cap} to be
\[
(C,\omega_C)=
\Big(
   P\times Q\times\C P^1\setminus D_{\Sigma},\;
   \omega_P+\omega_Q+\varrho^2\omega_{\fs}
\Big)\;.
\]
Applying Theorem \ref{thm:mainthm} to $(C,\omega_C)$
with the components of $D_{\Sigma}$
glued back
that correspond to the concave boundary components
we see that $(C,\omega_C)$
cannot have a convex boundary.
In other words
$(\Sigma,\alpha_{\Sigma})$
is the concave boundary of $(C,\omega_C)$.

\begin{rem}
  Let the dimension of $P\times Q\times\C$
  be $2n$.
  If the $(n-1)$-st power of the symplectic form
  $\omega=\omega_P+\omega_Q+\rmd x\wedge\rmd y$
  has a primitive $\mu$,
  as it is the case if $\omega_P$ is exact,
  the helicity (see \cite{mcsa04}) can be used as in \cite{zehzil13} to show
  that $\Sigma$ is the concave boundary of $(C,\omega_C)$.
  Indeed, by Stokes theorem the symplectic volume of
  $(D_{\Sigma},\omega)$
  equals $\int_{\Sigma}\mu\wedge\omega$,
  where $\Sigma$ is equipped with the boundary orientation
  induced by the symplectic orientation of $(D_{\Sigma},\omega)$.
  On the other hand,
  the restriction of $\omega$ to $T\Sigma$
  equals $\rmd\alpha_{\Sigma}$,
  where $\alpha_{\Sigma}$
  is the contact form induced by the local Liouville vector field $Y$,
  so that
  \[
  \Big(\mu|_{T\Sigma}-
  \alpha_{\Sigma}\wedge\big(\rmd\alpha_{\Sigma}\big)^{n-2}
  \Big)\wedge\rmd\alpha_{\Sigma}
  \]
  is an exact form on $\Sigma$.
  Consequently,
  $\int_{\Sigma}\mu\wedge\omega$
  equals the contact volume of $(\Sigma,\alpha_{\Sigma})$,
  so that
  \[
  i_Y\omega^n|_{T\Sigma}=
  n\alpha_{\Sigma}\wedge\big(\rmd\alpha_{\Sigma}\big)^{n-1}
  \]
  is a positive volume form on $(\Sigma,\alpha_{\Sigma})$.
  Hence, the boundary orientation on $\Sigma$
  equals the orientation induced by $\alpha_{\Sigma}$,
  i.e.\ the local Liouville vector field $Y$ points out.
\end{rem}

Consequently,
if $(A,\omega_A)$ is a symplectic cobordism
such that $(\Sigma,\alpha_{\Sigma})$
appears as convex boundary
gluing along $(\Sigma,\alpha_{\Sigma})$
yields a symplectic manifold
\[
(W,\omega_W)=(A,\omega_A)
\cup_{(\Sigma,\alpha_{\Sigma})}
(C,\omega_C)
\]
to which Theorem \ref{thm:mainthm} applies.
As an example we phrase.

\begin{cor}
\label{cor:geizehgen}
 The Gromov radius of
 \[
 \Big(
   P\times Q\times D^2,\;
   \omega_P+\omega_Q+\rmd x\wedge\rmd y
 \Big)
 \]
 is lower or equal than $\pi$.
 The strong Weinstein conjecture holds true
 for any closed contact type hypersurface
 in
 \[
 \Big(
   P\times Q\times\C,\;
   \omega_P+\omega_Q+\rmd x\wedge\rmd y
 \Big)
 \]
 provided that $P$ is not empty.
\end{cor}

In particular, we get the (very) stabilized strong Weinstein conjecture
for hypersurfaces of contact type in $Q\times\C^{\ell}$ for $\ell\geq2$,
cf.\ Floer-Hofer-Viterbo \cite{fhv90}.
If $H$ is a Donaldson hypersurface in $(Q,\omega_Q)$
(see \cite{don96})
so that the complement of $H$
has the structure of a Stein manifold $(P,\omega_P)$
then the strong Weinstein conjecture follows
for contact type hypersurfaces in $(Q\setminus H)\times\C$.
By \cite{agm01}
it is possible to construct symplectic hypersurfaces
$H$ that lie in the complement of a given
compact isotropic submanifold in $(Q,\omega_Q)$.

%%%%%%%%%%%%%%%%%%%%%%%%%%%%%%%%%%%%%%%%%%%%

\subsection{Cotangent bundles\label{subsec:cotanbu}}

We consider the unit sphere $S^{2m+1}$ in $\C^{m+1}$.
The Weinstein conjecture holds true
for any closed hypersurface $\Sigma$
that is of contact type in $T^*S^{2m+1}$,
cf.\ \cite{vit99}.
The contact structure on $\Sigma$ is taken to be the one
induced from $T^*S^{2m+1}$.
Notice, that no assumption is made
on the bounded component
of the complement $T^*S^{2m+1}\setminus\Sigma$
in view of the zero section,
cf.\ \cite{hovi89,suhzeh13}.
We claim that the strong Weinstein conjecture
holds for $\Sigma$ as well.
 
Indeed,
$S^{2m+1}$ embeds into $\C^{m+1}\times\C P^m$
as a Lagrangian submanifold $L$
via the map that sends $z\in S^{2m+1}\subset\C^{m+1}$
to $\big(z,[\bar{z}]\big)$.
The map $S^{2m+1}\ni w\mapsto [w]$,
where $[w]$ denotes the complex line through
$w$ and the origin,
is the so-called Hopf fibration,
along which the Fubini-Study form $\omega_{\fs}$
pulles back to $\rmd\mathbf{x}\wedge\rmd\mathbf{y}$.
Because the complex conjugation
$z\mapsto\bar{z}$ is anti-symplectic
the symplectic form
$\rmd\mathbf{x}\wedge\rmd\mathbf{y}+\omega_{\fs}$
on $\C^{m+1}\times\C P^m$ vanishes
pulled back to $S^{2m+1}$.
Hence, the $(2m+1)$-dimensional submanifold
$L\subset\C^{m+1}\times\C P^m$
is Lagrangian.
Using the fibrewise radial Liouville flow of $T^*S^{2m+1}$
we can isotope $\Sigma$ into a small neighborhood
of the zero-section.
Hence, with
Weinstein's tubular neighborhood theorem
(\cite{wein71})
$\Sigma$ can be realized as a contact type
hypersurface in $\C^{m+1}\times\C P^m$
with the characteristic foliation to be conjugate
to the one induced by the inclusion
$\Sigma\subset T^*S^{2m+1}$.
The claim follows with Corollary \ref{cor:geizehgen}.
 
In fact, this shows the strong Weinstein conjecture
for all cotangent bundles over closed manifolds
of the form $X\times S^{2m+1}$ with $m\geq0$,
which admit a Lagrangian embedding into
$T^*X\times\C^{m+1}\times\C P^m$.
To get products with $S^2$ observe that
$S^2$ embeds as a Lagrangian surface
into the unit ball in $\C^2$ blown-up at two points,
\cite[Section 6 Example (3)]{dgz14}.
To the blown-up ball the complement of the unit ball
in $\C\times (\C\cup\infty)$ is glued on.
With the construction of a symplectic cap
(the cap being $(\C\setminus B_1)\times\C P^1$)
and Theorem \ref{thm:mainthm} it follows
that any closed hypersurface of contact type
in the cotangent bundle of $X\times S^2$
satisfies the strong Weinstein conjecture.
 
Similarly, because for any closed orientable
$3$-manifold $Y$ the connected sum
$L=Y\#(S^1\times S^2)$
admits a Lagrangian embedding into $\C^3$
the strong Weinstein conjecture holds true
for $T^*L$,
see \cite{eems13}.
This is of particular interest for contact type
hypersurfaces in $T^*Y$
that miss one fibre.
Examples are given by
energy surfaces of
classical mechanical systems
on $Y$ with a sign changing
potential function,
cf.\ \cite{suhzeh13}.

%%%%%%%%%%%%%%%%%%%%%%%%%%%%%%%%%%%%%%%%
%%%%%%%%%%%%%%%%%%%%%%%%%%%%%%%%%%%%%%%%

\section{Stable curves and the moduli space\label{sec:stable}}

We consider a symplectic cobordism $(W,\omega)$
that admits an essential local foliation by holomorphic spheres
$B\times\C P^1\subset W$,
see Section \ref{subsec:essholfol}.
Denote by $*$ the particular base point of $B$
as required in Definition \ref{essdef}.
We denote by $(M,\alpha)$ the concave boundary of $(W,\omega)$
and associate the completion $(W',\omega_{\tau})$
attaching the negative half-symplectisation of $(M,\alpha)$
as described in Section \ref{subsec:compl}.
Moreover,
we assume the contact form $\alpha$ to be {\bf non-degenerate},
that is along periodic solutions of the Reeb vector field
of $\alpha$ the linearised Poincar\'e return map
has no eigenvalue equal to $1$.
Let $J$ be an admissible compatible almost complex structure
on $(W',\omega_{\tau})$ that is compatible
with the contact forms $\alpha$ and $\alpha_+$
on the concave end and in a neighborhood of $M_+$, resp.,
as described at the beginning of Section \ref{sec:exandappl}.
The aim of this section is to study stable holomorphic
one-marked curves of genus zero in $(W',\omega_{\tau},J)$.

%%%%%%%%%%%%%%%%%%%%%%%%%%%%%%%%%%%%%%%%%%%%

\subsection{Stable maps\label{subsec:stabmap}}

Let $(S,j,D)$ be a nodal Riemann surface of genus zero.
The set of points, the so-called {\bf nodal points},
that belong to a nodal pair is denoted by $|D|$.
We provide $(S,j,D)$ with a finite set $M$ of pairwise distinct
points in $S\setminus |D|$, the so-called {\bf marked points}.
The points in $|D|\cup M$ are called {\bf special}.
A connected component $C$ of $S$ is called {\bf stable}
if the number of special points in $C$ is greater or equal
than $3-2\genus(C)$.

A {\bf stable map} $(S,j,D,M,u)$ in $W'$
is a continuous map $u\co S\ra W'$
defined on a marked nodal Riemann surface $(S,j,D,M)$
that descends to a continuous map on the quotient $S/D$
such that:
\begin{itemize}
\item 
  The map $u$ is of Sobolev class $H_{\loc}^3$ on $S\setminus|D|$
  and of weighted Sobolev class $H^{3,\delta}$
  near the nodal points $|D|$ for some $\delta\in(0,2\pi)$,
  see \cite[Definition 1.1]{hwz-gw11}. 
\item 
  The cohomological integral $\int_Cu^*\omega_{\tau}$
  is non-negative for any connected component $C$ of $S$
  and for one $\tau\in\TT$
  (and hence for all by Stokes theorem).
\item 
  If a connected component $C$ of $S$ is {\it not} stable,
  then  $\int_Cu^*\omega_{\tau}>0$.
\end{itemize}

%%%%%%%%%%%%%%%%%%%%%%%%%%%%%%%%%%%%%%%%%%%%

\subsection{The moduli space\label{subsec:modulisp}}

Two stable maps $(S,j,D,M,u)$ and $(S',j',D',M',u')$
are said to be {\bf equivalent} if there exists a
diffeomorphism $\varphi\co S\ra S'$ such that
$\varphi^*j'=j$, $\varphi_*D=D'$, $\varphi M=M'$,
and $u'\circ\varphi=u$.
Often we will write $u$ for $(S,j,D,M,u)$.
The equivalence class $[S,j,D,M,u]$ of a stable map
is called a {\bf stable curve} and is denoted by $\mathbf u$.
We denote by
\[\ZZ\]
the set of all one-marked genus zero stable curves $\mathbf u$
in $W'$ such that the map induced by $u$ on the quotient $S/D$
is homologous to $*\times\C P^1$,
where $*\in B$ is the chosen base point of $B$.

A stable curve $\mathbf u$ is called {\bf holomorphic}
if it can be represented by a stable map $u$
that is holomorphic.
The definition does not depend on the choice of $u$.
We denote by
\[\MM\]
the {\bf moduli space}
of all holomorphic stable curves $\mathbf u\in\ZZ$.

Observe that for all $b\in B$ the class of
$\big(\C P^1,\rmi,\emptyset,\infty,z\mapsto (b,z)\big)$
represents an element of $\MM$ that we will call
a {\bf standard sphere}.
Due to the stability condition and the local foliation
$B\times\C P^1$ being essential
all non-standard holomorphic curves in $\MM$
do not intersect $B\times\C P^1$.
We will identify the subset of standard spheres
in $\MM$ with $B\times\C P^1$,
where the $\C P^1$-factor corresponds to the marked point.
The complement is denoted by
\[
\MM_{\cut}=\MM\setminus B\times\C P^1\;.
\]

%%%%%%%%%%%%%%%%%%%%%%%%%%%%%%%%%%%%%%%%%%%%

\subsection{A priori uniform bounds\label{subsec:bounds}}

The Dirichlet-energy of all nodal holomorphic spheres $u$
induced by stable holomorphic curves $\mathbf u\in\MM$
equals
\[
\int_Su^*\omega_{\tau}=
[\omega_{\tau}]\Big(\big[\C P^1\big]\Big)=
\pi\varrho^2
\]
for all $\tau\in\TT$.

The moduli space $\MM$ admits upper bounds in the following sense.
Because the almost complex structure $J$ is compatible
with the contact form $\alpha_+$ 
the maximum principle implies that no stable holomorphic curve
can intersect a neighborhood $(\varepsilon,0]\times M_+$
of $M_+$,
cf.\ the first part of the proof of Proposition \ref{prop:essential}.
Similarly, because only standard curves $\mathbf u\in\MM$ intersect
the foliation domain $B\times\C P^1$
we can bound $\MM_{\cut}$ away from $\partial B$
if the boundary of $B$ is not empty.

As it will turn out (see Section \ref{subsec:gwint})
there are no lower bounds for $\MM$
along the concave end.
This will prove Theorem \ref{thm:mainthm}
as the following lemma shows:

\begin{lem}
\label{lem:firstcase}
 If there exists a sequence $\mathbf u_k$
 in $\MM$ such that $u_k(S_k)$ intersects
 $(-\infty,-k]\times M$ non-trivially,
 then $(M,\alpha)$ admits a null-homologous
 Reeb link of total action less than $\pi\varrho^2$.
\end{lem}

\begin{proof}
 If $u_k$ maps a component $C$ of $S_k$
 into the concave end
 $u_k|_C$ is constant
 by Stokes theorem.
 Therefore, we find sequences $z_k$ and $w_k$
 of points on $S_k$
 such that $u_k(z_k)\in W$ and
 $u_k(w_k)$ is contained in the concave end
 so that the projection of $u_k(w_k)$ to the $\R$-factor
 tends to $-\infty$.
 With respect to a metric on $W'$ that equals
 a product metric on the concave end
 (product with the Euclidean metric on the $\R$-axis)
 the gradient of $u_k$ blows up.
 This follows with a mean value argument
 as in \cite{geizeh13,hwz03}.
 The bubbling off analysis from \cite{behwz03} shows
 the existence of a holomorphic building
 of height $k_-|1$ for some $k_-\geq1$.
 The lowest level of the building is represented by a
 punctured finite energy surface
 in the symplectisation of $(M,\alpha)$
 that has Hofer-energy less than $\pi\varrho^2$
 and positive punctures exclusively.
 Near the punctures the finite energy surface
 converges to cylinders over closed Reeb orbits
 of $\alpha$ exponentially fast.
 Hence, the projection of a component
 to $M$ along the $\R$-axis defines a $2$-chain
 whose boundary is a Reeb link in $(M,\alpha)$
 of total action less than $\pi\varrho^2$,
 cf.\ \cite{geizeh12,hofe93}.
\end{proof}

In other words,
in order to prove Theorem \ref{thm:mainthm}
we have to exclude the case
where there exists $K>0$ such that for all
$\mathbf u$ in $\MM$ the image $u(S)$
is contained in $(-K,0]\times M\cup W$.

We remark that the proof of Lemma \ref{lem:firstcase}
uses the assumption $\alpha$ being non-degenerate.
This is not a restriction as the arguments in Section \ref{sec:proof}
will show.

%%%%%%%%%%%%%%%%%%%%%%%%%%%%%%%%%%%%%%%%
%%%%%%%%%%%%%%%%%%%%%%%%%%%%%%%%%%%%%%%%

\section{Polyfold structure\label{sec:polystr}}

%%%%%%%%%%%%%%%%%%%%%%%%%%%%%%%%%%%%%%%%%%%%

\subsection{Topology of the space of stable curves\label{subsec:topofspstabcu}}

The space $\ZZ$ of stable curves has a natural topology
as described in \cite[Section 2.1/3.4]{hwz-gw11}
that is second countable, paracompact, and
Hausdorff, see \cite[Theorem 1.6]{hwz-gw11}.
The topology is induced by the $H^3$-topology
of maps on nodal Riemann surfaces
that have exponential decay in holomorphic polar coordinates
near the nodes.
Part of the construction
is a choice of auxiliary marked points
that stabilize all connected components of the domain
if necessary.
The additional points are fixed using local
codimension-$2$ submanifolds in $W'$
that are transverse to the image of the curve
intersecting in a single point.
It is required that the additional marked points
are mapped to the intersection points.
If a non-trivial automorphism group
is acting on a stable curve the auxiliary
marked points are supposed to be chosen equivariantely.
The stabilization of the domain makes it possible
to use the topology of the corresponding
Deligne-Mumford space in terms of uniformizing families,
see \cite[Definition 2.12]{hwz-gw11}.
In order to describe the desingularization
of the nodes (i.e.\ the gluing)
uniformizing families are used to obtain
uniformizers for the space of stable maps $\ZZ$,
see \cite[Section 3.1/3.2]{hwz-gw11}.

We remark that the evaluation map
$\ev\co\ZZ\ra W'$ that maps $\mathbf u$
to the value $u(z)$ at the marked point $z$
is continuous,
see \cite[p.~2290]{hwz-III09} or \cite[p.~7]{hwz-gw11}.

%%%%%%%%%%%%%%%%%%%%%%%%%%%%%%%%%%%%%%%%%%%%

\subsection{The target space\label{subsec:target}}

Let $u$ be a stable map
that represents a class $\mathbf u$ in $\ZZ$.
We denote by $\xi$ a continuous section of
$\Hom\big(\Lambda T^*S,u^*TW'\big)$
such that for each $z\in S$
the map $\xi(z)\co T_zS\ra T_{u(z)}W'$
is complex anti-linear with respect to $J\big(u(z)\big)$.
The section $\xi$ is required to be of
Sobolev class $H^2_{\loc}$ on $S\setminus|D|$
and of weighted Sobolev class $H^{2,\delta}$ near $|D|$
for some $\delta\in(0,2\pi)$,
see \cite[Section 1.2]{hwz-gw11}.
An equivalence $\varphi$ of stable maps
$(S,j,D,M,u)$ and $(S',j',D',M',u')$
is an {\bf equivalence} of
$(S,j,D,M,u,\xi)$ and $(S',j',D',M',u',\xi')$
if $\xi'\circ T\varphi=\xi$.
The equivalence class is denoted by $\boldxi$
and the space of all equivalence classes $\boldxi$
by
\[
\WW
\;.
\]
By \cite[Theorem 1.9]{hwz-gw11}
$\WW$ has a natural topology
that is second countable, paracompact, and
Hausdorff so that the projection $p\co\WW\ra\ZZ$
that maps the class $\boldxi$ to the class $\mathbf u$,
if $\xi$ is a section along $u$,
is continuous.
The {\bf Cauchy-Riemann operator} $\CR{J}$ is a section of $p$
whose value $\CR{J}\mathbf u$ at a point $\mathbf u\in\ZZ$
is the class represented by
\[
\Big(
  S,j,D,M,u,\tfrac12\big(Tu+J(u)\circ Tu\circ j\big)
\Big)
\;.
\]
Notice that the moduli space $\MM$ equals the zero set
$\{\CR{J}\mathbf u=\mathbf 0\}$.

%%%%%%%%%%%%%%%%%%%%%%%%%%%%%%%%%%%%%%%%%%%%

\subsection{Polyfold Fredholm section\label{subsec:polfredsec}}

The space of stable curves $\ZZ$
is a polyfold (see \cite[Section 3]{hwz-III09})
so that the evaluation map
$\ev\co\ZZ\ra W'$ is sc-smooth,
see \cite[Theorem 1.7/1.8]{hwz-gw11},
\cite[Theorem 1.10]{hwz-III09}.
The projection map $p\co\WW\ra\ZZ$
is a strong polyfold bundle,
see \cite[Theorem 1.10]{hwz-gw11}.
By \cite[Theorem 1.11]{hwz-gw11}
the Cauchy-Riemann operator
$\CR{J}\co\ZZ\ra\WW$
is a sc-smooth, component proper
Fredholm section which is naturally oriented.
The Fredholm index of
$\CR{J}$ equals the dimension of $W$
because the first Chern class of $(TW',J)$
evaluates to $2$ on $\big[*\times\C P^1\big]$.

On the level of objects
the strong polyfold bundle structure of
$p\co\WW\ra\ZZ$ is obtained
by gluing strong polyfold bundles
in the sense of ep-groupoids.
To understand the Cauchy-Riemann section
$\CR{J}\co\ZZ\ra\WW$
it is enough to consider a
local M-polyfold bundle chart.
We would like to apply this principle
for standard holomorphic spheres
$\mathbf u\in\MM$.
Represent $\mathbf u$ by the map
$u\co z\mapsto (b,z)$ for some $b\in B$.
Denote by $X$ the sc-Hilbert manifold
of pairs $(v,z)$,
where $z\in\C P^1$ is a marked point near $\infty$
and $v$ a map $\C P^1\ra W'$ of class $H^3$
that is close to $u$ mapping $0$ and $1$
into $B\times 0$ and $B\times 1$, resp.
Moreover,
let $E$ be the sc-Hilbert space bundle over $X$
with fibre consisting of all $H^2$-maps
from $\C P^1$ into the space of $J(v)$-anti-linear
maps $\Lambda T^*\C P^1\ra v^*TW'$
for $v\in X$.
A M-polyfold bundle chart of
$p\co\WW\ra\ZZ$ about $\mathbf u$
is then given by $E\ra X$
because $u$ is un-noded,
so that the surrounding splicing core is the full ambient space
and the sc-retraction map equals the identity.

The Cauchy-Riemann operator
$\CR{J}\co\ZZ\ra\WW$
induces a local section
$f\co X\ra E$ near $x=(u,\infty)$.
The {\bf linearization} $f'(x)\co T_xX\ra E_x$
is the vertical differential,
which is the composition of the sc-differential
$T_xf\co T_xX\ra T_0E$
and the projection
$T_0E=T_xX\oplus E_x\ra E_x$,
see \cite[Section 4.4]{hwz-I07}.
By \cite[p.~55 and Definition 5.5]{hwz-gw11} and
\cite[Section 3]{hwz-II09} the linearization $f'(x)$
is a sc-Fredholm operator
and coincides with the linearized Cauchy-Riemann
operator at $(u,\infty)$
in the $C^{\infty}$-sense (cf.\ \cite{mcsa04})
on a dense subset,
see \cite[Proposition 2.14/2.15]{hwz-I07}.
Because a sc-Fredholm operator
is regularizing $f'(x)$ is surjective
with kernel being of dimension equal to $\dim W$
that consists of smooth elements exclusively.

%%%%%%%%%%%%%%%%%%%%%%%%%%%%%%%%%%%%%%%%%%%%

\subsection{Abstract Perturbations\label{subsec:abstrpert}}

Let $\lambda\co\WW\ra\Q\cap [0,\infty)$ denote a
sc$^+$-multisection.
In a local representation $P\co E\ra X$ of $p\co\WW\ra\ZZ$
we write $\Lambda\co E\ra\Q\cap [0,\infty)$
for the sc$^+$-multisection that corresponds to $\lambda$.
The local sc$^+$-sections of $P$
that are attached to $\Lambda$
are denoted by $s_i$.
A pair $\big(\CR{J},\lambda\big)$ is called {\bf transverse}
if in local representations $f\co X\ra E$
of the Cauchy-Riemann operator
the linearizations
\[
f'(u)-s'_i(u)\co T_uX\lra E_u
\]
are surjective for all $i$
and for all $\mathbf u\in\ZZ$
that are contained in
$\big\{\lambda\big(\CR{J}\big)>0\big\}$,
see \cite[Definition 4.7(1)]{hwz-III09}.
Observe that in case $f'(u)$ is onto
for $\mathbf u\in\MM$ that has a simple representation
by an un-noded holomorphic sphere map $u$
we can choose one local section that is identically $0$
in a neighborhood of $u$ in $X$,
i.e.\ $\lambda(0_u)=1$.

%%%%%%%%%%%%%%%%%%%%%%%%%%%%%%%%%%%%%%%%%%%%

\subsection{Gromov-Witten integration\label{subsec:gwint}}

In this section we give a proof of
Theorem \ref{thm:mainthm}
in the case the contact form $\alpha$
is non-degenerate.
For $K>0$ denote by $W_K$
the open domain in $W'$
that is obtained from $W'$ by removing
$(-\infty,-K]\times M$, $[-1/K,0]\times M_+$,
and $B_K\times\C P^1$,
where $B_K$ is a subset of $B$ that is
diffeomorphic to either a ball of radius $1/K$
or, if $\partial B$ is not empty, a collar neighborhood
$[-1/K,0]\times\partial B$ of $\partial B$.
We have to exclude uniform lower bounds for $\MM$,
see Lemma \ref{lem:firstcase}.

Arguing by contradiction we assume that there exists $K>0$
such that for all
non-standard curves $\mathbf u\in\MM_{\cut}$
the image $u(S)$ is contained in $W_K$.
There exists a neighborhood $\UU$ of $\MM_{\cut}$
in $\ZZ$ such that for all $\mathbf u\in\UU$
the image $u(S)$ is contained in $W_K$,
cf.\ Section \ref{subsec:topofspstabcu}.
For example we can choose $\UU=\ev^{-1}(W_K)$
because of $\ev(\MM_{\cut})\subset W_K$.
Moreover, by \cite{behwz03} and
\cite[Proposition 4.9 and Remark 4.10]{hwz-gw11}
$\MM$ is a compact subset of $\ZZ$.

By \cite[Theorem 4.17]{hwz-III09}
there exists a small sc$^+$-multisection $\lambda$
such that $\big(\CR{J},\lambda\big)$
is transverse.
Notice, that for all standard spheres
$\mathbf u\in\MM$ the isotropy group of
$\mathbf u$ is trivial and $f'(u)$ is onto,
where $f$ is a local representation of $\CR{J}$.
The proof of \cite[Theorem 4.17]{hwz-III09}
implies that $\lambda$ can be chosen to be
trivial for all standard spheres $\mathbf u\in\MM$,
i.e.\ in a local representation $\Lambda$
is identically $1$ on the zero-section
over the set of standard holomorphic spheres $u$,
cf.\ \cite[Definition 3.35]{hwz-III09}.
The support of $\lambda$
can be assumed to be contained in $\UU$.
Therefore, the solution set
\[
\SS=
\Big\{
\mathbf u\in\ZZ
\;\Big|\;
\lambda\big(\CR{J}\mathbf u\big)>0
\Big\}
\]
of $\big(\CR{J},\lambda\big)$
is an oriented compact branched suborbifold
of dimension $\dim W$ with boundary $\partial \SS$,
see \cite[Theorem 4.17]{hwz-III09}
and \cite[Section 1.4]{hwz-gw11}.
Moreover,
$\SS\setminus\big(B\times\C P^1\big)$
is contained in $\UU$
and $\SS$ is equipped with the weight function
\[
\vartheta\co\ZZ\lra\Q\cap [0,\infty)\,;\\
\quad\mathbf u\longmapsto\lambda\big(\CR{J}\mathbf u\big)\;.
\]
A neighborhood of $\partial\SS$ in $\SS$,
which is equal to a neighborhood of $\partial\MM$
in $\MM$,
can be identified with a collar neighborhood
of $\partial B\times\C P^1$ in
$B\times\C P^1\subset W$.

In order to reach the desired contradiction
let $\Omega$ be a top-dimensional differential form on $W'$
of total volume $\int_W\Omega=1$
that has support in a sufficiently small open ball
which we require to have closure contained
in the interior of $B_K\times\C P^1$.
Notice,
that the evaluation map
$\ev\co\SS\ra W'$
induces an embedding of
$B_K\times\C P^1$ into  the solution space $\SS$,
which coincides with the moduli space $\MM$
along the image
$B_K\times\C P^1$.
Therefore,
\[
1=\int_W\Omega=
\int_{(\SS,\vartheta)}\ev^*\Omega\;.
\]
Using a compactly supported diffeotopy
in the interior of $W'$
we can isotope the support of $\Omega$
into the complement of
$W_K\cup\big(B_K\times\C P^1\big)$.
Denoting by $\Omega_1$ the image of $\Omega$
under pull back along the diffeotopy
the difference $\Omega-\Omega_1$
has a primitive $\mu$
with compact support in the interior of $W'$.
Because the support of $\Omega_1$
does not lie in the image of the evaluation map
$\ev\co\SS\ra W'$ the restriction of $\ev^*\Omega_1$
to $\SS$ vanishes.
By Stokes theorem \cite[Theorem 1.27]{hwz-int10}
\[
\int_{(\SS,\vartheta)}\ev^*\Omega=
\int_{(\partial\SS,\vartheta)}\ev^*\mu\;.
\]
But $\ev^*\mu$ restricted to the boundary $\partial\SS$
must vanish because $\mu$ has compact support in
the complement of $\partial W'$.
This is a contradiction
that finishes the proof of 
Theorem \ref{thm:mainthm}
for a non-degenerate contact form $\alpha$.

With the same argument one shows that if $(W,\omega)$
has no concave boundary
$(W,\omega)$ cannot have a convex boundary,
either.

%%%%%%%%%%%%%%%%%%%%%%%%%%%%%%%%%%%
%%%%%%%%%%%%%%%%%%%%%%%%%%%%%%%%%%%

\section{Proof of Theorem \ref{thm:mainthm}\label{sec:proof}}

We claim
that there exist $N\in\N$ and periodic solutions
$x_1,\ldots,x_N$
of the Reeb vector field $R$ of $\alpha$
that are of (not necessary prim) period
$T_1,\ldots,T_N$, resp.,
such that
$T_1+\ldots+T_N<\pi\varrho^2$
and
$[x_1]+\ldots+[x_N]=0$
in the homology of $M$.
With Section \ref{subsec:gwint}
and Lemma \ref{lem:firstcase}
the claim follows for a non-degenerate
contact form $\alpha$.
If $\alpha$ is degenerate
we find a sequence of
contact forms $\alpha_k$ on $M$
that are non-degenerate
such that $\alpha_k$ tends to $\alpha$
in $C^{\infty}$,
see \cite[Proposition 6.1]{hwz98}.
Observe,
that the Reeb vector field $R_k$
of $\alpha_k$ tends to $R$ in the
$C^{\infty}$-topology too.

First of all
we consider a sequence of periodic solutions
$x_k$ of $R_k$ that are of period
$T_k<\pi\varrho^2$.
To the reparametrised sequence
$y_k(t)=x_k(T_kt)$
the Arzel\`a-Ascoli theorem applies
so that a subsequence of $y_k$
converges in $C^{\infty}(\R/\Z)$
to a loop $y$ in $M$
that is tangent to $\R R$
and has action
\[
T:=\int_0^1y^*\alpha
=\lim_{k\ra\infty}\int_0^1y_k^*\alpha_k
\leq\pi\varrho^2\;,
\]
because $\int_0^1y_k^*\alpha_k=T_k$.
The loop $x(t):=y(t/T)$ is a $T$-periodic solution of $R$.
To express this circumstance
we will say that a subsequence of $x_k$
converges to $x$ after reparametrisation.

With this preliminaries
we apply the established existence result
to the non-degenerate contact form
$\alpha_k$ for each $k$.
This results in a sequence of
periodic solutions
$x^k_1,\ldots,x^k_{N_k}$
of $R_k$ that are of period
$T^k_1,\ldots,T^k_{N_k}$, resp.,
such that the periods
sum up to total acton less than
$\pi\varrho^2$ and the loops
represent the zero class
in the homology of $M$.
By the flow-box theorem applied to $R$
we get $A_k>A/2$ for $k$ sufficiently large,
where $A$ and $A_k$ denote
the minimal action of a periodic solution
of $R$ and $R_k$, resp.
Hence, $N_k$ is bounded from above
by $2\pi\varrho^2/A$
so that we can assume
the number of link components
to be constantly equal to $K\in\N$.
Therefore, we find a subsequence
$x^k_1,\ldots,x^k_K$
that converges after reparametrisation
to periodic solutions
$x_1,\ldots,x_K$
of $R$ with period $T_1+\ldots+T_K\leq\pi\varrho^2$
such that 
$[x^k_1]=[x_1],\ldots,[x^k_K]=[x_K]$ for
$k$ sufficiently large.
In particular,
the sum $[x_1]+\ldots+[x_K]$ must vanish.
The desired claim follows
with the exception of a non-strict
inequality for the total action
of the null-homologous Reeb link.
In order to obtain the strict inequality
observe that we can realize
$\big(M,(1+\delta)\alpha\big)$
as the concave boundary of $(W,\omega)$
for some small $\delta>0$.
\hfill
Q.E.D.

\begin{rem}{\bf (Weak contact type boundary condition)}
\label{rem:weak}
 In Theorem \ref{thm:mainthm}
 the convex boundary can be replaced by a
 $J${\bf -convex boundary} in the sense of
 \cite{elia90} or \cite{mcd91},
 cf.\ \cite[Section 3.2 (C4)]{geizeh12} or \cite{mnw13}.
 This means the boundary components $M_+$
 oriented as the boundary of $(W,\omega)$
 admit a positive contact form $\alpha_+$
 such that in a neighborhood of $M_+$
 there exists an almost complex structure $J$
 satisfying the following properties:
 \begin{itemize}
 \item 
   $J$ is tamed by $\omega$,
 \item 
   $J$ leaves the kernel $\xi_+$ of $\alpha_+$ invariant,
   i.e.\ $\xi_+=TM\cap JTM$, and
 \item 
   $J$ restricted to the contact structure $\xi_+$
   is tamed by $\rmd\alpha_+$.
 \end{itemize}
 The reason is that there exists a collar neighborhood $U$
 of $M$ such that any $J$-holomorphic sphere
 that intersects $U$ must be constant.
 Indeed, $U$ can be chosen to be equal to
 $(-\varepsilon,0]\times M$ for $\varepsilon>0$
 sufficiently small
 such that the restricton of $\partial_a$ to the boundary
 $M$ equals $-JR_+$, where $R_+$
 is the Reeb vector field of $\alpha_+$.
 Shrinking $\varepsilon>0$ if necessary
 and setting $K=1/\varepsilon$
 the symplectic form $\rmd(\rme^{Ka}\alpha_+)$
 tames $J$ on $U$ and $-\rmd(\rmd\rme^{Ka}\circ J)$
 is positive on $J$-complex lines,
 cf.\ \cite[Remark 4.3]{geizeh10}.
\end{rem}

\begin{proof}[{\bf Proof of Corollary \ref{cor:gwstatement}}]
 In the case $(W,\omega)$ has no contact type boundary
 the proof of Theorem \ref{thm:mainthm} shows that
 the evaluation map restricted to the solution space $\SS$
 is surjective.
 Sending the abstract perturbations to zero
 one obtains
 a family of corresponding nodal solutions
 through each point of $W$
 that admit converging subsequences in the topology of $\ZZ$,
 cf.\ \cite[Theorem 4.17]{hwz-III09}.
 The limits are holomorphic nodal spheres.
\end{proof}

%%%%%%%%%%%%%%%%%%%%%%%%%%%%%%%%%%%%%%%%

\begin{ack}
  The research of this article was carried
  out while the authors stay at the {\it Institut for Advanced Studies} Princeton.
  We would like to thank Helmut Hofer for the invitation
  and for introducing us to the beauty of polyfolds.
  We thank Peter Albers, Matthew Strom Borman, Joel Fish,
  and Hansj\"org Geiges for enlightening discussions.
  We thank Chris Wendl and the referees for valuable
  comments on the first manuscript.
\end{ack}

%%%%%%%%%%%%%%%%%%%%%%%%%%%%%%%%%%%%%%%%%%%%%%%%%%%%%%%%%%%%%%%%%%%%%%

\end{document}